\documentclass[10pt]{article}
 \font\smallit=cmti10

\usepackage{amssymb,latexsym,amsmath,epsfig,amsthm} 

\makeatletter

\renewcommand{\@seccntformat}[1]{\csname the#1\endcsname. }

\makeatother

\newtheorem{theorem}{Theorem}[section]
 \newtheorem{lemma}[theorem]{Lemma}
 
 \newtheorem{proposition}[theorem]{Proposition}
 \newtheorem{corollary}[theorem]{Corollary}
 \newtheorem{definition}[theorem]{Definition}
 
\begin{document}
\begin{center}
{\bf On $p$-adic Gram-Schmidt Orthogonalization Process}
 \vskip 30pt

{\bf Yingpu Deng}\\
 {\smallit  Key Laboratory of Mathematics Mechanization, NCMIS, Academy of Mathematics and Systems Science, Chinese Academy of Sciences, Beijing 100190, People's Republic of China}\\
{and}\\
 {\smallit School of Mathematical Sciences, University of Chinese Academy of Sciences, Beijing 100049, People's Republic of China}\\


 \vskip 10pt

 {\tt dengyp@amss.ac.cn}\\

 \end{center}
 
 \vskip 30pt
 
 \centerline{\bf Abstract} In his famous book ``Basic Number Theory", Weil proved several theorems about the existence of norm-orthogonal bases in finite-dimensional vector spaces and lattices over local fields. In this paper, we transform Weil's proofs into algorithms for finding out various norm-orthogonal bases. These algorithms are closely related to the recently introduced closest vector problem (CVP) in $p$-adic lattices and they have applications in cryptography based on $p$-adic lattices.
 
\vspace{0.5cm}



\noindent

 \pagestyle{myheadings}

 \thispagestyle{empty}
 \baselineskip=12.875pt
 \vskip 20pt

\section{Introduction}
Let $p$ be a prime. Let $V$ be a left vector space over $\mathbb{Q}_p$. A norm $N$ on $V$ is a function
$$N:V\longrightarrow\mathbb{R}$$
such that:
\begin{enumerate}
\item[(i)] $N(v)\geq0,\forall v\in V, \mbox{ and }N(v)=0\mbox{ if and only if }v=0;$
\item[(ii)] $N(xv)=|x|_p\cdot N(v),\forall x\in\mathbb{Q}_p, v\in V;$
\item[(iii)] $N(v+w)\leq\max(N(v),N(w)), \forall v,w\in V.$
\end{enumerate}
Here, $|x|_p$ is the $p$-adic absolute value for a $p$-adic number $x\in\mathbb{Q}_p$.

If $N$ is a norm on $V$, and if $N(v)\neq N(w)$ for $v,w\in V$, then we must have $N(v+w)=\max(N(v),N(w))$. Weil \cite[p.26, prop.3]{wei} proved the following proposition:

\begin{proposition}\cite{wei}\label{orthogonal base}
	Let $V$ be a left vector space over $\mathbb{Q}_p$ of finite dimension $n>0$, and let $N$ be a norm on $V$. Then there is a decomposition $V=V_1+\cdots+V_n$ of $V$ into a direct sum of subspaces $V_i$ of dimension 1, such that
	$$N\left (\sum_{i=1}^nv_i\right)=\max_{1\leq i\leq n}N(v_i), \forall v_i\in V_i, i=1,\ldots,n.$$
\end{proposition}

Weil proved the above proposition for finite-dimensional vector spaces over a $p$-field $K$ (commutative or not). For simplicity, we only consider the case $K=\mathbb{Q}_p$. So we can define the $N$-orthogonal basis.
\begin{definition}[$N$-orthogonal basis]
	Let $V$ be a left vector space over $\mathbb{Q}_p$ of finite dimension $n>0$, and let $N$ be a norm on $V$.  We call $\alpha_1,\ldots,\alpha_n$  an $N$-orthogonal basis of $V$ over $\mathbb{Q}_p$ if  $V$ can be decomposed  into the direct sum of $n$ 1-dimensional subspaces $V_i$'s $(1\leq i\leq n)$, such that
	$$N\left(\sum_{i=1}^nv_i\right)=\max_{1\leq i\leq n}N(v_i), \forall v_i\in V_i, i=1,\ldots,n,$$
	where  $V_i$ is spanned by $\alpha_i$. Two subspaces $U,W$ of $V$ is said to be $N$-orthogonal if the sum $U+W$ is a direct sum and it holds that $N(u+w)=\max(N(u),N(w))$ for all $u\in U,w\in W$.
\end{definition}

 Weil's proof is not constructive, and he did not give an algorithm to find out an $N$-orthogonal basis. In recent years, new computational problems in $p$-adic lattices are introduced in \cite{deng2} (also see \cite{deng1}), and these computational problems can be applied to construct cryptographic schemes (see \cite{deng3}). Norm-orthogonal basis plays a crucial role in the construction of these cryptographic schemes (see \cite{deng3}). So, it is natural to ask whether there are algorithms to find out norm-orthogonal bases.
 
 In the next section, we revisit Weil's proof of Proposition \ref{orthogonal base} and in section \ref{gram-schmidt} we will transform Weil's proof into an algorithm for finding out an $N$-orthogonal basis. In section \ref{simul}, we give an algorithm for finding out an $N$-, $N'$-orthogonal basis for two norms simultaneously. In section \ref{laorth}, we consider the orthogonalization process in $p$-adic lattices.

\section{Weil's proof of Proposition \ref{orthogonal base}}
We now revisit Weil's proof of Proposition \ref{orthogonal base}, see \S 1 of Chapter II of Weil's book \cite{wei} (pages 24-27). Using induction on the dimension $n$ of $V$ over $\mathbb{Q}_p$, the case $n=1$ is obvious. Suppose $n>1$. Let $u_1,u_2,\ldots,u_n$ be a $\mathbb{Q}_p$-basis of $V$. Let $W$ be the subspace of $V$ of dimension $n-1$ spanned by $u_2,\ldots,u_n$. Define a linear form
$$f:V\longrightarrow \mathbb{Q}_p,$$
$$f(a_1u_1+a_2u_2+\cdots+a_nu_n)=a_1.$$
Then we have Ker$(f)=W$. Weil proved that a nonzero vector $v_1\in V$ which is $N$-orthogonal to $W$ can be put a maximum point of the continuous function
$$\frac{|f(v)|_p}{N(v)}$$
on the compact subset $I:=\{v\in V:p^{-1}\leq N(v)\leq1\}$ of $V$. In the next section, we give a procedure for finding out such a maximum point $v_1$.

\section{$p$-adic orthogonalization process}\label{gram-schmidt}
Keeping the notation in the previous section. We first need a simple lemma.

\begin{lemma}\label{1basis}
Let $V$ be a left vector space over $\mathbb{Q}_p$, and let $N$ be a norm on $V$. Let $U$ be a subspace of $V$ of dimension 1. Then there is a nonzero vector $u\in U$ with $p^{-1}<N(u)\leq1$.
\end{lemma}

\begin{proof}
Take any nonzero vector $v\in U$. Then there is an integer $m\in\mathbb{Z}$ with $p^{m-1}<N(v)\leq p^m$. Put $u=p^mv$.
\end{proof}

Let $n>1$. By induction on $W$, there is an $N$-orthogonal basis $v_2,\ldots,v_n$ on $W$. By Lemma \ref{1basis}, we can assume $N(u_1),N(v_2),\ldots,N(v_n)\in(p^{-1},1]$.

For an arbitrary element $v$ of $V$, write as $v=a_1u_1+a_2v_2+\cdots+a_nv_n$ with $a_i\in\mathbb{Q}_p$ for $1\leq i\leq n$. Put $w=a_2v_2+\cdots+a_nv_n\in W$, then $v=a_1u_1+w$. The continuous function is
$$\frac{|f(v)|_p}{N(v)}=\frac{|a_1|_p}{N(v)}.$$
Since we want to find a maximum point $v\in I$ of this function, we must have $a_1\neq0$. So the problem turns into finding out a minimum point $v\in I$ with $a_1\neq0$ of the function $N(v/a_1)$.

Suppose $v\in I$. Then we have $p^{-1}\leq N(v)\leq1$. Consider the values of $N(a_1u_1),N(w)$ according to three cases: $<p^{-1},[p^{-1},1],>1$. By the fact, if $N(a_1u_1)\neq N(w)$, then we have $N(v)=N(a_1u_1+w)=\max\{N(a_1u_1),N(w)\}$, it is easy to see that there are only four cases to deal with:

Case 1: $N(a_1u_1)\in[p^{-1},1]$ and $N(w)<p^{-1}$.

Case 2: $N(a_1u_1)<p^{-1}$ and $N(w)\in[p^{-1},1]$.

Case 3: $N(a_1u_1)\in[p^{-1},1]$ and $N(w)\in[p^{-1},1]$.

Case 4: $N(a_1u_1)>1$ and $N(w)>1$.

\begin{proposition}
In Case 1: $N(a_1u_1)\in[p^{-1},1]$ and $N(w)<p^{-1}$, we have
$$N\left(\frac{v}{a_1}\right)=N(u_1)\leq1$$
is constant.
\end{proposition}

\begin{proof}
Since $N(w)<N(a_1u_1)$, we have $N(v)=N(a_1u_1+w)=N(a_1u_1)$. The result follows.
\end{proof}

\begin{proposition}
In Case 2: $N(a_1u_1)<p^{-1}$ and $N(w)\in[p^{-1},1]$, we have
$$N\left(\frac{v}{a_1}\right)\geq1.$$
\end{proposition}

\begin{proof}
Since $N(a_1u_1)<N(w)$, we have $N(v)=N(a_1u_1+w)=N(w)$. Since $p^{-1}>N(a_1u_1)=|a_1|_p\cdot N(u_1)>p^{-1}|a_1|_p$, we have $|a_1|_p<1$, i.e., $|a_1|_p\leq p^{-1}$. Thus
$$N\left(\frac{v}{a_1}\right)=\frac{N(w)}{|a_1|_p}\geq\frac{p^{-1}}{p^{-1}}=1.$$
\end{proof}

\textbf{Remark.} Compare this proposition with the previous one, from the view of point of finding out a minimum point of $N(v/a_1)$, we need not to consider this case.

\begin{proposition}
In Case 3: $N(a_1u_1)\in[p^{-1},1]$ and $N(w)\in[p^{-1},1]$, we need only to consider $|a_1|_p=1$ and $a_i\in\mathbb{Z}_p$ for $2\leq i\leq n$, hence
$$\frac{v}{a_1}=u_1+\sum_{i=2}^n\frac{a_i}{a_1}\cdot v_i\in u_1+\sum_{i=2}^n\mathbb{Z}_p\cdot v_i.$$
\end{proposition}

\begin{proof}
Since $N(w)=\max_{2\leq i\leq n}N(a_iv_i)\leq1$, we have $N(a_iv_i)\leq1$ for $2\leq i\leq n$. From $1\geq N(a_iv_i)=|a_i|_p\cdot N(v_i)>p^{-1}|a_i|_p$, we get $|a_i|_p<p$ if $a_i\neq0$. So we get $|a_i|_p\leq1$, i.e., $a_i\in\mathbb{Z}_p$ for $2\leq i\leq n$. Similarly, we have $|a_1|_p\leq1$. From $p^{-1}\leq N(a_1u_1)=|a_1|_p\cdot N(u_1)\leq|a_1|_p$, we get $p^{-1}\leq|a_1|_p\leq 1$. Hence $|a_1|_p=p^{-1}$ or 1. Obviously, if $N(u_1)=1$, we have $|a_1|_p=p^{-1}$ or 1; if $N(u_1)<1$, we have $|a_1|_p=1$.

If $|a_1|_p=p^{-1}$, we have
$$N\left(\frac{v}{a_1}\right)=\frac{N(v)}{p^{-1}}\geq1$$
since $N(v)\geq p^{-1}$. From the view of point of finding out a minimum point of $N(v/a_1)$, we need not to consider this case.

If $|a_1|_p=1$, we get the case stated in the proposition.
\end{proof}

\begin{proposition}
In Case 4: $N(a_1u_1)>1$ and $N(w)>1$, we have $a_i/a_1\in\mathbb{Z}_p$ for $2\leq i\leq n$, hence
$$\frac{v}{a_1}=u_1+\sum_{i=2}^n\frac{a_i}{a_1}\cdot v_i\in u_1+\sum_{i=2}^n\mathbb{Z}_p\cdot v_i.$$
\end{proposition}

\begin{proof}
Since $N(v)=N(a_1u_1+w)\leq1$, we must have $N(a_1u_1)=N(w)$. Since $N(w)=\max_{2\leq i\leq n}N(a_iv_i)$, we have $N(a_iv_i)\leq N(a_1u_1)$ for $2\leq i\leq n$. Thus $|a_i|_p\cdot p^{-1}<|a_i|_p\cdot N(v_i)\leq N(a_1u_1)=|a_1|_p\cdot N(u_1)\leq|a_1|_p$ if $a_i\neq0$. So we get
$$\left|\frac{a_i}{a_1}\right|_p<p.$$
I.e., $|\frac{a_i}{a_1}|_p\leq1$, hence the result of the proposition.
\end{proof}

Put the above four propositions together, we get the main auxiliary result.

\begin{lemma}\label{mainlemma}
The problem of finding out a nonzero vector $v_1\in V$ which is $N$-orthogonal to the hyperplane $W$ can be reduced to solving the closest vector problem (CVP) of the $p$-adic lattice $\mathcal{L}=\sum_{i=2}^n\mathbb{Z}_p\cdot v_i$ with the target vector $u_1\notin\mathcal{L}$. The CVP can be solved by a deterministic algorithm which terminates within finitely many steps if we can compute efficiently the norm $N(v)$ of any vector $v\in V$.
\end{lemma}

\begin{proof}
As we have seen that, the problem of finding out a nonzero vector $v_1\in V$ which is $N$-orthogonal to the hyperplane $W$ can be reduced to finding out a minimum point $v\in I$ with $a_1\neq0$ of the function $N(v/a_1)$. From the above four propositions, this problem can be reduced to solving the closest vector problem (CVP) of the $p$-adic lattice $\mathcal{L}=\sum_{i=2}^n\mathbb{Z}_p\cdot v_i$ with the target vector $u_1\notin\mathcal{L}$. If we find a closest vector $w_0\in\mathcal{L}$ to the target vector $u_1$, by Lemma \ref{1basis}, we can find a vector $v_1\in I$ and $a_1\in\mathbb{Q}_p$ with $a_1\neq0$ such that $v_1/a_1=u_1-w_0$. As we already know that $v_1$ is then $N$-orthogonal to the hyperplane $W$.

The another statement of the lemma follows from Theorems 4.4, 4.5 and 4.6 in \cite{deng2}. Note that the norm $N$ considered in \cite{deng2} is the $p$-adic absolute value of an extension field of $\mathbb{Q}_p$, and all results in \cite{deng2} obviously hold also for any norm $N$ if we can compute efficiently the norm $N(v)$ of any vector $v\in V$.
\end{proof}

We can now prove the main theorem.

\begin{theorem}
Let $V$ be a left vector space over $\mathbb{Q}_p$ of finite dimension $n>0$, and let $N$ be a norm on $V$. Let $u_1,u_2,\ldots,u_n$ be a $\mathbb{Q}_p$-basis of $V$. Then there is a deterministic algorithm which terminates within finitely many steps for finding out an $N$-orthogonal basis of $V$ if we can compute efficiently the norm $N(v)$ of any vector $v\in V$.
\end{theorem}

\begin{proof}
Induction on $n$. If $n=1$, there is nothing to prove.
 Suppose $n>1$. Let $W$ be the subspace of $V$ of dimension $n-1$ spanned by $u_2,\ldots,u_n$. By induction hypothesis on $W$, there is an $N$-orthogonal basis $v_2,\ldots,v_n$ on $W$. By Lemma \ref{mainlemma}, there is a deterministic algorithm for finding out a nonzero vector $v_1\in V$ which is $N$-orthogonal to the hyperplane $W$. Then $v_1,v_2,\ldots,v_n$ is an $N$-orthogonal basis of $V$.
\end{proof}

\section{Simultaneous $p$-adic orthogonalization process}\label{simul}
Weil \cite[p.26, prop.4]{wei}  also proved the following proposition:
\begin{proposition}\cite{wei}\label{w2}
	Let $V$ be a left vector space over $\mathbb{Q}_p$ of finite dimension $n>0$, and let $N,N'$ be two norms on $V$. Then there is a decomposition $V=V_1+\cdots+V_n$ of $V$ into a direct sum of subspaces $V_i$ of dimension 1, such that
	$$N\left (\sum_{i=1}^nv_i\right)=\max_{1\leq i\leq n}N(v_i),$$
 $$N'\left (\sum_{i=1}^nv_i\right)=\max_{1\leq i\leq n}N'(v_i),\forall v_i\in V_i, i=1,\ldots,n.$$
\end{proposition}

In other words, there is a $\mathbb{Q}_p$-basis $v_1,\ldots,v_n$ of $V$ which is $N$-orthogonal and is also $N'$-orthogonal. Weil did not give an algorithm to find out such a basis. Below, we turns Weil's proof into an algorithm for finding out such a basis. First, we review Weil's proof of Proposition \ref{w2}.

\subsection{Weil's proof of Proposition \ref{w2}}

Induction on $n$. If $n=1$, this is clear. For $n>1$, first find a nonzero vector $v_1\in V$ such that $v_1$ is a maximum point of the continuous function
$$\frac{N(v)}{N'(v)}$$
on the compact subset $I:=\{v\in V:p^{-1}\leq N(v)\leq1\}$ of $V$. Then find a hyperplane $W$ which is $N$-orthogonal to $v_1$. Weil proved that $W$ is also $N'$-orthogonal to $v_1$. By induction hypothesis on $W$, there is a $\mathbb{Q}_p$-basis $v_2,\ldots,v_n$ of $W$ which is $N$-orthogonal and is also $N'$-orthogonal. Thus the $\mathbb{Q}_p$-basis $v_1,v_2,\ldots,v_n$ of $V$ satisfies the requirement of the proposition, i.e., it is $N$-orthogonal and is also $N'$-orthogonal.

\subsection{An algorithmic version}

If we want to turn the above Weil's proof into an algorithm for simultaneous $N$- and $N'$-orthogonal basis, we have two things to do: firstly, giving an algorithm for find $v_1$, secondly, giving an algorithm for find a hyperplane $W$ which is $N$-orthogonal to $v_1$.

\begin{lemma}\label{findhyperplane}
Let $V$ be a left vector space over $\mathbb{Q}_p$ of finite dimension $n>0$, and let $N$ be a norm on $V$. Given a nonzero vector $v_1\in V$ and a $\mathbb{Q}_p$-basis $u_1,\ldots,u_n$ of $V$. There is a deterministic algorithm which terminates within finitely many steps for finding out a hyperplane $W$ which is $N$-orthogonal to $v_1$ if we can compute efficiently the norm $N(v)$ of any vector $v\in V$.
\end{lemma}

\begin{proof}
Induction on $n$. If $n=1$, put $W=\{0\}$. For $n>1$, write $v_1$ as $v_1=a_1u_1+\cdots+a_nu_n$ with $a_i\in\mathbb{Q}_p$ for $1\leq i\leq n$. Since $v_1$ is nonzero, not all $a_i$ are zero. Without loss of generality, we may assume $a_n\neq0$. Then $u_1,\ldots,u_{n-1},v_1$ is a $\mathbb{Q}_p$-basis of $V$. Set $V'$ be the subspace of $V$ spanned by $u_2,\ldots,u_{n-1},v_1$. By induction on $V'$, we can find a hyperplane $W'$ in $V'$ which is $N$-orthogonal to $v_1$. By Lemma \ref{mainlemma}, we can find out a nonzero vector $u_1'\in V$ which is $N$-orthogonal to $V'$. Now, let $W$ be spanned by $u_1'$ and $W'$. Then the hyperplane $W$ is $N$-orthogonal to $v_1$.
\end{proof}

Now we can prove the following:

\begin{theorem}
Let $V$ be a left vector space over $\mathbb{Q}_p$ of finite dimension $n>0$, and let $N,N'$ be two norms on $V$. Given a $\mathbb{Q}_p$-basis $u_1,\ldots,u_n$ of $V$. There is a deterministic algorithm which terminates within finitely many steps for finding out a $\mathbb{Q}_p$-basis $v_1,\ldots,v_n$ of $V$ which is $N$-orthogonal and is also $N'$-orthogonal if we can compute efficiently the norms $N(v),N'(v)$ of any vector $v\in V$.
\end{theorem}

\begin{proof}
Induction on $n$. If $n=1$, this is clear. For $n>1$, we first find a nonzero vector $v_1\in V$ such that $v_1$ is a maximum point of the continuous function
$$\frac{N(v)}{N'(v)}$$
on the compact subset $I:=\{v\in V:p^{-1}\leq N(v)\leq1\}$ of $V$ by the following method. By Lemma \ref{1basis}, we can assume $N(u_i)\in(p^{-1},1]$ for $1\leq i\leq n$. For $v\in I$, write $v$ as $v=a_1u_1+\cdots+a_nu_n$ with $a_i\in\mathbb{Q}_p$ for $1\leq i\leq n$. Since $v$ is nonzero, not all $a_i$ are zero. Set $|a_j|_p=\max_{1\leq i\leq n}|a_i|_p$. Then $a_j\neq0$. So $v/a_j$ is a sum of $u_j$ and $\mathbb{Z}_p$-linear combination of other $n-1$ basis vectors. Let $\mathcal{L}$ be the lattice spanned by $\{u_i:1\leq i\leq n,i\neq j\}$. We have $v/a_j=u_j+$a lattice vector of $\mathcal{L}$. Since $u_j\notin\mathcal{L}$, we know from Section 4 in \cite{deng2} that both $N(u_j+\mathcal{L})$ and $N'(u_j+\mathcal{L})$ take only finitely many positive values. Since
$$\frac{N(v)}{N'(v)}=\frac{N(v/a_j)}{N'(v/a_j)},$$
so $N(v)/N'(v)$ takes only finitely many positive values. And from Section 4 in \cite{deng2} we know that there is a deterministic algorithm which terminates within finitely many steps to compute these positive values and corresponding lattice vectors if we can compute efficiently the norms $N(v),N'(v)$ of any vector $v\in V$. Letting $j=1,2,\ldots,n$, we can get a nonzero vector $v_1\in V$ such that $v_1$ is a maximum point of the continuous function
$$\frac{N(v)}{N'(v)}$$
on the compact subset $I:=\{v\in V:p^{-1}\leq N(v)\leq1\}$ of $V$.

By Lemma \ref{findhyperplane}, there is a deterministic algorithm to find out a hyperplane $W$ which is $N$-orthogonal to $v_1$. Using induction hypothesis on $W$, we can find out a $\mathbb{Q}_p$-basis $v_2,\ldots,v_n$ of $W$ which is $N$-orthogonal and is also $N'$-orthogonal. Now the $\mathbb{Q}_p$-basis $v_1,v_2,\ldots,v_n$ of $V$ which is $N$-orthogonal and is also $N'$-orthogonal. We are done.
\end{proof}

\section{Orthogonalization of $p$-adic lattices}\label{laorth}
We have dealt with the orthogonalization process of bases in vector spaces over $\mathbb{Q}_p$, in this section, we consider the orthogonalization of bases in $p$-adic lattices. We first recall the definition of a $p$-adic lattice.

\begin{definition}
Let $V$ be a left vector space over $\mathbb{Q}_p$ of finite dimension $n>0$, and let $N$ be a norm on $V$. Let $v_1,\ldots,v_m(1\leq m\leq n)$ be $\mathbb{Q}_p$-linearly independent vectors of $V$. The subset
$$\mathcal{L}=\mathcal{L}(v_1,\ldots,v_m):=\left\{\sum_{i=1}^m a_i\cdot v_i: a_i\in\mathbb{Z}_p,i=1,\ldots,m\right\}$$
of $V$ is called a $p$-adic lattice in $V$. $m$ is the rank of the lattice and $v_1,\ldots,v_m$ is a basis of the lattice.
\end{definition}

Lattices are compact subsets of $V$. The following proposition can be found in \cite[p.72, prop.]{rob}.

\begin{proposition}
Let $\Omega\subset V$ be a compact subset. (a) For every $a\in V-\Omega$, the set of norms $\{N(x-a):x\in\Omega\}$ is finite. (b) For every $a\in\Omega$, the set of norms $\{N(x-a):x\in\Omega-\{a\}\}$ is discrete in $\mathbb{R}_{>0}$.
\end{proposition}

Due to this proposition, we can define the longest vector problem (LVP) and the closest vector problem (CVP) in $p$-adic lattices in this general setting as in \cite{deng2}, and all results in \cite{deng2} obviously hold also for any norm $N$ if we can compute efficiently the norm $N(v)$ of any vector $v\in V$.

\begin{definition}
Let $V$ be a left vector space over $\mathbb{Q}_p$ of finite dimension $n>0$, and let $N$ be a norm on $V$. If $\alpha_1,\ldots,\alpha_m$ is an $N$-orthogonal basis of the vector space spanned by a lattice $\mathcal{L}=\sum_{i=1}^{m}\mathbb{Z}_p\alpha_i$, then we call $\alpha_1,\ldots,\alpha_m$ an $N$-orthogonal basis of the lattice $\mathcal{L}$.
\end{definition}

Weil \cite[p.29, th.1]{wei} proved the following result.

\begin{proposition}\cite{wei}
Let $V$ be a left vector space over $\mathbb{Q}_p$ of finite dimension $n>0$, and let $N$ be a norm on $V$. Let $\mathcal{L}=\sum_{i=1}^{m}\mathbb{Z}_p\alpha_i$ be a lattice in $V$. Then there is a norm $N'$ on $V$ and an $N'$-orthogonal basis $\beta_1,\ldots,\beta_m$ such that $\mathcal{L}=\sum_{i=1}^{m}\mathbb{Z}_p\beta_i$.
\end{proposition}

In general, the norm $N'$ is different from the norm $N$. In \cite[p.28, prop.6]{wei}, Weil proved that $N'$ can be taken by
$$N'(v)=\inf_{x\in\mathbb{Q}_p^{\times},xv\in\mathcal{L}}|x|_p^{-1},\mbox{ for }v\in V.$$
So, a natural question is whether there is an $N$-orthogonal basis of the lattice $\mathcal{L}=\sum_{i=1}^{m}\mathbb{Z}_p\alpha_i$ given a basis $\alpha_1,\ldots,\alpha_m$ of $\mathcal{L}$. When the rank $m$ is 1, the answer is clear. The main result of this section is that the answer is affirmative if the rank $m$ is two, and we give a deterministic algorithm to find out an $N$-orthogonal basis given an arbitrary basis of a rank-2 lattice.

\begin{lemma}\label{twosum}
Let $V$ be a left vector space over $\mathbb{Q}_p$, and let $N$ be a norm on $V$. Let $v,w\in V$. Then we have $N(v+w)=\max(N(v),N(w))$ if and only if $N(v+w)\geq N(v)$.
\end{lemma}

\begin{proof}
The 'only if' part of the lemma is clear. We prove 'if' part as follows. Suppose $N(v+w)\geq N(v)$. Hence $N(w)=N(v+w-v)\leq\max(N(v+w),N(v))=N(v+w)$. So $N(v+w)\geq\max(N(v),N(w))$. By the definition of a norm, we have $N(v+w)\leq\max(N(v),N(w))$. Therefore $N(v+w)=\max(N(v),N(w))$.
\end{proof}

\begin{corollary}\label{pointtolattice}
Let $V$ be a left vector space over $\mathbb{Q}_p$, and let $N$ be a norm on $V$. Let $\alpha_1,\ldots,\alpha_n(n>1)$ be $\mathbb{Q}_p$-linearly independent vectors of $V$. Set $\mathcal{L}=\mathcal{L}(\alpha_2,\ldots,\alpha_n)$. Then we have:
$$N(\alpha_1+w)=\max(N(\alpha_1),N(w))\mbox{ for all }w\in\mathcal{L}$$
if and only if
$$N(\alpha_1)=\min\{N(\alpha_1+w):w\in\mathcal{L}\}.$$
\end{corollary}

\begin{proof}
This follows immediately from Lemma \ref{twosum}.
\end{proof}

\textbf{Remark.} By solving a CVP-instance, for the lattice $\mathcal{L}=\mathcal{L}(\alpha_2,\ldots,\alpha_n)$ and the target vector $\alpha_1$, we can find $w_0\in\mathcal{L}(\alpha_2,\ldots,\alpha_n)$ such that
$$N(\alpha_1+w_0)=\min\{N(\alpha_1+w):w\in\mathcal{L}\}.$$
There is a deterministic algorithm to find such a $w_0$ (see \cite{deng2}) if we can compute efficiently the norm $N(v)$ of any vector $v\in V$.

\begin{lemma}\label{critoorth}
Let $V$ be a left vector space over $\mathbb{Q}_p$ of finite dimension $n>1$, and let $N$ be a norm on $V$. Let $\alpha_1,\ldots,\alpha_n$ be $\mathbb{Q}_p$-linearly independent vectors of $V$. Then $\alpha_1,\ldots,\alpha_n$ is an $N$-orthogonal basis of $V$ if and only if it holds that
$$N\left(\sum_{i=1}^na_i\alpha_i\right)=\max_{1\leq i\leq n}N(a_i\alpha_i)$$
where one of the $a_1,\ldots,a_n$ is 1 and the remaining $n-1$ many ones are in $\mathbb{Z}_p$.
\end{lemma}

\begin{proof}
The 'only if' part of the lemma is clear. We prove 'if' part as follows. Suppose the condition in the lemma is satisfied. We want to prove that $\alpha_1,\ldots,\alpha_n$ is an $N$-orthogonal basis of $V$, i.e., we have
$$N\left(\sum_{i=1}^na_i\alpha_i\right)=\max_{1\leq i\leq n}N(a_i\alpha_i)$$
for all $a_i\in\mathbb{Q}_p(1\leq i\leq n)$. If $a_1,\ldots,a_n$ are all zero, this is clear. Suppose $a_1,\ldots,a_n$ are not all zero. There is an index $j\in\{1,\ldots,n\}$ such that $|a_j|_p=\max_{1\leq i\leq n}|a_i|_p$. Then $a_j\neq0$ and $a_i/a_j\in\mathbb{Z}_p$ for all $1\leq i\leq n$. Thus
$$N\left(\sum_{i=1}^na_i\alpha_i\right)=|a_j|_p\cdot N\left(\sum_{i=1}^n\frac{a_i}{a_j}\alpha_i\right)=|a_j|_p\cdot\max_{1\leq i\leq n}N\left(\frac{a_i}{a_j}\alpha_i\right)=\max_{1\leq i\leq n}N(a_i\alpha_i)$$
using the condition in the lemma. We are done.
\end{proof}

Now we can prove the main result of this section.

\begin{theorem}
Let $V$ be a left vector space over $\mathbb{Q}_p$, and let $N$ be a norm on $V$. Let $\alpha,\beta$ be two $\mathbb{Q}_p$-linearly independent vectors of $V$. Let $\mathcal{L}=\mathcal{L}(\alpha,\beta)$ be a lattice of rank 2 in $V$. Then there is a deterministic algorithm to find out an $N$-orthogonal basis of the lattice $\mathcal{L}$ if we can compute efficiently the norm $N(v)$ of any vector $v\in V$.
\end{theorem}

\begin{proof}
By Corollary \ref{pointtolattice}, we can find $w_0\in\mathcal{L}(\beta)$ such that $N(\alpha+w_0)=\min\{N(\alpha+w):w\in\mathcal{L}(\beta)\}$. Set $\alpha'=\alpha+w_0$. Similarly, we can find $v_0\in\mathcal{L}(\alpha')$ such that $N(v_0+\beta)=\min\{N(v+\beta):v\in\mathcal{L}(\alpha')\}$. Set $\beta'=v_0+\beta$. We claim that $\alpha',\beta'$ is an $N$-orthogonal basis of the lattice $\mathcal{L}=\mathcal{L}(\alpha,\beta)$.

Firstly, it is clear that we have $\mathcal{L}(\alpha,\beta)=\mathcal{L}(\alpha',\beta)=\mathcal{L}(\alpha',\beta')$. It remains to prove that $\alpha',\beta'$ is $N$-orthogonal. By Corollary \ref{pointtolattice}, we have $N(v+\beta')=\max(N(v),N(\beta'))$ for all $v\in\mathcal{L}(\alpha')$. 

Below, we prove $N(\alpha'+w)=\max(N(\alpha'),N(w))$ for all $w\in\mathcal{L}(\beta')$. We write $w=k\beta'(k\in\mathbb{Z}_p)$ and $v_0=l_0\alpha'(l_0\in\mathbb{Z}_p)$. Then
$$\alpha'+w=\alpha'+k\beta'=\alpha'+k(l_0\alpha'+\beta)=(1+kl_0)\alpha'+k\beta.$$
We distinguish two cases:

Case 1: $k\in p\mathbb{Z}_p$. Then $1+kl_0\in\mathbb{Z}_p^{\times}$. We have 
$$N(\alpha'+w)=N\left(\alpha'+\frac{k}{1+kl_0}\beta\right)\geq N(\alpha')$$
by the definition of $\alpha'$. By Lemma \ref{twosum}, we have $N(\alpha'+w)=\max(N(\alpha'),N(w))$.

Case 2: $k\in\mathbb{Z}_p^{\times}$. We have
$$N(\alpha'+w)=N(\alpha'+k\beta')=N\left(\frac{1}{k}\alpha'+\beta'\right)\geq N(\beta')=N(k\beta')=N(w)$$
by the definition of $\beta'$. By Lemma \ref{twosum}, we have $N(\alpha'+w)=\max(N(\alpha'),N(w))$.

Thus, we always have $N(v+\beta')=\max(N(v),N(\beta'))$ for all $v\in\mathcal{L}(\alpha')$ and $N(\alpha'+w)=\max(N(\alpha'),N(w))$ for all $w\in\mathcal{L}(\beta')$. By Lemma \ref{critoorth}, $\alpha',\beta'$ is $N$-orthogonal. We are done.
\end{proof}

\textbf{Question.} Is this true for lattices with rank greater than or equal to three? We leave it as an open problem.

\vspace{0.5cm}

\textbf{Acknowledgments}: We thank Dr. Zhaonan Wang and Dr. Chi Zhang for helpful discussion. This work was supported by National Natural Science Foundation of China (No. 12271517) and National Key Research and Development Project of China (No. 2018YFA0704705).


\begin{thebibliography}{99}
\bibitem{deng1} Yingpu Deng, Lixia Luo and Guanju Xiao, \textit{On Some Computational Problems in Local Fields}, Cryptology ePrint Archive, Report 2018/1229. http: //eprint.iacr.org/2018/1229.

\bibitem{deng2} Yingpu Deng, Lixia Luo, Yanbin Pan and Guanju Xiao, \textit{On Some Computational Problems in Local Fields}, Journal of Systems Science and Complexity, \textbf{35}(2022), 1191--1200.

\bibitem{deng3} Yingpu Deng, Lixia Luo, Yanbin Pan, Zhaonan Wang and Guanju Xiao, \textit{Public-key Cryptosystems and Signature Schemes from $p$-adic Lattices}, Cryptology ePrint Archive, Report 2021/522. http: //eprint.iacr.org/2021/522.

\bibitem{rob} A.M. Robert, \textit{A course in $p$-adic analysis}, GTM 198, Springer, New York, 2000.

\bibitem{wei} A. Weil, \textit{Basic number theory}, Third edition, Springer, New York, 1974.

\end{thebibliography}
\end{document}